\documentclass[11pt]{amsart}

\usepackage{amsmath}
\usepackage{amssymb}
\usepackage{epsfig}
\usepackage{amsfonts}
\usepackage{ytableau}
\ytableausetup{boxsize=1.2em}

\newcommand{\excise}[1]{}

\newtheorem{thm}{Theorem}[section]
\newtheorem{lemma}[thm]{Lemma}

\newtheorem{cor}[thm]{Corollary}

\newtheorem{conv}[thm]{Convention}

\newtheorem{rem}[thm]{Remark}
\newtheorem{Warn}[thm]{Caution}

    {\end{Example}}
    {\end{Remark}}

\def\emp{\nothing}
\def\sq{\square}

\def\nn{\mathbb N}
\def\cc{\mathbb C}

\def\fq{{\mathbb F}_q}
\def\ov{\overline}

\def\sm{\smallsetminus}

\def\Ga{\Gamma}

\def\la{\lambda}

\def\de{\delta}

\def\al{\alpha}

\def\vp{\varphi}

\def\cB{{\mathcal {B}}}

\def\cA{{\mathcal {A}}}

\def\cH{{\mathcal {H}}}

\def\cP{\mathcal P}

\def\cQ{\mathcal Q}

\def\ssu{\subset}

\def\<{\langle}
\def\>{\rangle}

\def\GL{ {\text {\rm GL} } }

\def\0{{\mathbf 0}}

\def\nothing{\varnothing}

\def\.{\hskip.06cm}
\def\ts{\hskip.03cm}

\def\lra{\leftrightarrow}

\def\corn{{{\text{\rm $v$}}}}

\title[Unimodality via Kronecker products]{Unimodality via Kronecker products}

\author[Igor~Pak]{ \ Igor~Pak$^\star$}

\author[Greta~Panova]{ \ Greta~Panova$^\star$}

\date{\today}

\thanks{\thinspace ${\hspace{-.45ex}}^\star$Department of Mathematics, UCLA, Los Angeles, CA 90095, \ts
\texttt{\{pak,panova\}@math.ucla.edu}}

\begin{document}
\maketitle

\begin{abstract}
We present new proofs and generalizations of unimodality of the
$q$-binomial coefficients~$\binom{n}{k}_q$ as polynomials in~$q$.
We use an algebraic approach by interpreting the differences between numbers
of certain partitions as Kronecker coefficients of representations of~$S_n$.
Other applications of this approach include strict unimodality of the diagonal
$q$-binomial coefficients and unimodality of certain partition statistics.
\end{abstract}

%
%


\section{Introduction}\label{sec:intro}

\noindent
A sequence $(a_1,a_2,\ldots,a_n)$ is called \emph{unimodal}, if for some~$k$ we have
$$
a_1 \, \le \, a_2 \, \le \,  \ldots \, \le \,  a_k \, \ge \,  a_{k+1}
\, \ge \, \ldots \, \ge \, a_n\ts.
$$
The study of unimodality of combinatorial sequences is a classical subject
going back to Newton, and has intensified in recent decades.  There is
a remarkable diversity of applicable tools, ranging from analytic to topological,
and from representation theory to probabilistic analysis.  The results have
a number of application, but are also important in their own right.  We refer
to~\cite{B1,B2,Sta-unim} for a broad overview of the subject.

\smallskip

In this paper we present two extensions of the following classical
unimodality result.  The \emph{$q$-binomial} (Gaussian) \emph{coefficients}
are defined as:
$$
\binom{m+\ell}{m}_q
\, = \ \. \frac{(q^{m+1}-1)\. \cdots\. (q^{m+\ell}-1)}{(q-1)\.\cdots\. (q^{\ell}-1)}
\ \. = \, \, \sum_{n=0}^{\ell\ts m} \, \. p_n(\ell,m) \. q^n\ts.
$$
The unimodality of a sequence
$$p_0(\ell,m)\ts, \, p_1(\ell,m)\ts, \, \ldots \,, \, p_{\ell\ts m}(\ell,m)
$$
is a celebrated result first conjectured by Cayley in~1856,
and proved by Sylvester in~1878~\cite{Syl} (see also~\cite{Sta-Lie}).
Historically, it has been a starting point of many investigations
and various generalizations, both of combinatorial and algebraic nature,
and the problem remains very difficult.  We refer to Section~\ref{s:fin} 
for discussion of various proofs, connections with the Sperner property, 
historical remarks and references.  

\smallskip

Recall that $p_n(\ell,m) = \# \cP_n(\ell,m)$, where
$\cP_n(\ell,m)$ is the set of partitions $\la\vdash n$, such that
$\la_1 \le m$ and~$\la'_1\le \ell$.  Denote by~$\corn(\la)$
the number of distinct part sizes in the partition~$\la$.
The sequence $(a_1,\ldots,a_n)$ is called \emph{symmetric}~if $a_i = a_{n+1-i}$,
for all $i\le i \le n$.


\begin{thm}\label{t:qbin}
Let
$$
p_n(\ell,m,r) \. = \. \sum_{\la \in \cP_n(\ell,m)} \. \binom{\corn(\la)}{r}\ts.
$$
Then the sequence
$$p_r(\ell,m,r), \ts p_{r+1}(\ell,m,r), \ldots, p_{\ell\ts m}(\ell,m,r)
$$
is symmetric and unimodal.
\end{thm}

\smallskip

Note that  $p_n(\ell,m,r)=0$ for $n<\binom{r+1}{2}$ or $n>\ell\ts m -\binom{r}{2}$,
and that~$\corn(\la)$
can be viewed as the \emph{number of corners} of the corresponding
Young diagram~$[\la]$.
Moreover, $p_n(\ell,m,0) = p_n(\ell,m)$ and  therefore, for $r=0$,
Theorem~\ref{t:qbin} gives the unimodality of $q$-binomial coefficients.
Our next theorem is a different extension of this result in the diagonal case.

\smallskip

\begin{thm}\label{t:strict}
Let $a_n = p_n(m,m)$.  Then, for all $\.m\ge 7\ts$, we have:
$$
a_{1}\. < \. a_{2}\. < \. \ldots \. < \. a_{\lfloor m^2/2\rfloor}
\. = \. a_{\lceil m^2/2\rceil} \. > \. \ldots \. > a_{m^2-2}\. > \. a_{m^2-1}\ts.
$$
\end{thm}

\smallskip
Of course, the new contributions of this theorem are the strict inequalities
(see also Subsection~\ref{ss:fin-post}).
The idea behind the proof of Theorem~\ref{t:qbin} is to consider tensor
products $S^{\lambda}\otimes S^{\nu}$
of irreducible representations of~$S_n$, where $\nu=(n-k,k)$ is a two-row
partition.  We study the \emph{Kronecker coefficients}
$g(\la,\mu,\nu)$, defined as the multiplicity of $S^{\nu}$ in the tensor product representation $S^{\lambda}\otimes S^{\nu}$, namely
\begin{equation}\label{def:kron}
S^{\ts\lambda}\otimes S^{\ts\mu} \, = \, \scalebox{1.3}{$\underset{\nu\vdash n}{\oplus} $}\,\, g(\la,\mu,\nu) \. S^{\ts\nu}\ts
\end{equation}
and interpret these coefficients combinatorially, as the difference in the number of
certain \emph{Littlewood--Richardson} (LR) \emph{tableaux}.  We then prove that
these tableaux are in bijection with the desired partitions.
The inequality $g(\la,\mu,\nu)\ge 0$ then implies unimodality.

The proof of Theorem~\ref{t:strict} is more intricate and uses further ingredients.
We employ the main lemma in~\cite{PPV}  to show that $g(\la,\mu,\nu) >0$ and thereby to reduce strict positivity of
Kronecker coefficients to \emph{strict unimodality} of sufficiently large
coefficients of a polynomial
$$
\cA_m(q) \, = \, \prod_{i=1}^m \. \bigl(1+q^{2\ts i-1}\bigr), \qquad \text{for all \ $m\ge~27$\ts.}
$$
To prove this result (Theorem \ref{t:prod}), we strengthen Almkvist's proof of (non-strict) unimodality
of $\cA_m(q)+q+q^{m^2-1}$, see~\cite{A1}.

\smallskip

The paper is structured as follows. We start with definitions and notations
in Section~\ref{s:def}.  We then present the Main Lemma on
unimodality of certain products of LR~coefficients (Section~\ref{s:main}).
In sections~\ref{s:app} and~\ref{s:strict}, we apply the Main Lemma to derive
all theorem~\ref{t:qbin} and~\ref{t:strict}, respectively.  In the following
Section~\ref{s:dual}, we present a dual version of the Main Lemma and derive
algebraically a weak version of Almkvist's theorem.  We conclude with final
remarks and open problems in Section~\ref{s:fin}.

\medskip

\section{Definitions, notation and examples}\label{s:def}

\noindent
We refer the reader to~\cite{Mac,Sta} for the background on symmetric
functions and combinatorics of Young tableaux. Here we set the notations,
recall the LR~rule, and include an example of Theorem~\ref{t:qbin}.

\smallskip

\subsection{Partitions and Young diagrams}\label{ss:def-part}
For any integer partition $\pi=(\pi_1,\ldots,\pi_k)$ let $\pi'$ denote its~\emph{conjugate},
i.e.~the partition whose Young diagram~$[\pi']$ is the transpose of the Young diagram of $\pi$,
or algebraically $\pi'_i = \# \{j: \pi_j\geq i\}$. Let $(a^b) = (a,\ldots,a)$, $b$~times,
denote the partition whose shape is a $b\times a$ rectangle.
Assuming there is a fixed rectangle $(a^b)$ in the context, we denote by $\bar{\pi}$
the \emph{complement} of $\pi$ within this rectangle, i.e. $\bar{\pi}_i = a-\pi_{b+1-i}$.
For example, if $\pi=(5,5,3,2)$, then $\pi'=(4,4,3,2,2)$, the complement of $\pi$ within
the $(6^4)$ rectangle is $\bar{\pi} = (4,3,1,1)$ (we assume that $\pi_j=0$ for $j>k$).

\smallskip

\subsection{Symmetric functions and the Kronecker product} \label{ss:def-sym}
Following~\cite{Mac,Sta}, we use $e_k$ and $h_k$ to denote elementary and
homogeneous symmetric functions, respectively, and let $s_\la$ be the Schur functions.
We use ``$*$'' to denote the \emph{Kronecker product} in the ring
of symmetric functions, so
$$s_\la \ts * \ts s_\mu \, = \, \sum_{\nu\vdash n} \. g(\la,\mu,\nu) \. s_\nu\..
$$
Here $g(\la,\mu,\nu)$ are the Kronecker coefficients as defined by \eqref{def:kron} in \S\ref{sec:intro}. Unlike the Littlewood-Richardson coefficients, explained in $\S$\ref{ss:def-LR}, which have many nice properties like an easy combinatorial interpretation, no such properties are present for the Kronecker coefficients in the general case. The best current result along these lines is the combinatorial interpretation of Blasiak~\cite{Bla} in the case when one of the partitions $\la,\mu,\nu$ is a hook. 

\smallskip

\subsection{The LR rule}\label{ss:def-LR}
The LR coefficients $c^{\lambda}_{\mu\nu}$ are originally defined as the multiplicity
of the irreducible representation $V_{\lambda}$ of $\GL(N,\cc)$ within the tensor product
$V_{\mu}\otimes V_{\nu}$. For our purposes we will recall their original combinatorial
interpretation in terms of \emph{semi-standard Young tableaux} (SSYT).

The \emph{reading word} of a semi-standard Young tableaux $T$ is the sequence obtained
by successively recording the numbers appearing in $T$ starting from the top row to
the bottom row and reading each row from right to left. A \emph{lattice permutation}
(ballot sequence) is a sequence of positive integers $w=w_1w_2\ldots w_n$,
such that, for every $k$ and $i$, among the first $k$ letters of $w$ there are at
least as many $i$'s as $(i+1)$'s, or formally
$$
\#\{j: w_j=i, j\leq k\} \, \geq  \, \. \#\{ j: w_j=i+1,j\leq k\} \qquad \text{for all} \quad 1\le k \le n\ts, \ i\ge 1\ts.
$$
We say that a sequence or a tableau is of \emph{type}~$\beta$ if it has $\beta_i$ numbers
equal to~$i$.

The \emph{Littlewood--Richardson rule} states that $c^{\lambda}_{\mu\nu}$
is equal to the number of SSYT's of shape~$\lambda/\mu$, of type~$\nu$, and
whose reading word is a lattice permutation. We call such tableaux the
\emph{Littlewood--Richardson} (LR) \emph{tableaux}.

For example, if $\lambda=(5,5,3,2)$, $\mu=(2,1)$ and $\nu =(4,4,3,1)$, then the semi-standard
tableau~$X$ below
is an LR~tableau of shape $\lambda/\mu$, type~$\nu$, and whose reading word is $111222133243$.
$$\ytableausetup{centertableaux}
X \ = \ \ \ytableaushort{\none\none 111,\none 1222,233,34}
$$

\smallskip

\subsection{Partitions in a rectangle} \label{ss:def-rect}
Let $\ell=m=3$.  Then $\cP_n=\cP_n(3,3)$ are as follows (here, for brevity and aesthetics, we use a concise notation for the partitions, e.g. instead of $(3,2,2)$ we write $32^2$):
$$
\cP_0=\emp, \
\cP_1=\{1\}, \
\cP_2=\{2,1^2\}, \
\cP_3=\{3,21,1^3\}, \  \cP_4=\{31,22,21^2\},
$$
$$
\cP_5=\{32,221,31^2\}, \,
\cP_6=\{3^2,321,2^3\}, \,
\cP_7=\{3^21,32^2\}, \,
\cP_8=\{3^22\}, \,
\cP_9=\{3^3\}\ts.
$$
Therefore,
$$
\binom{6}{3}_q \, = \, \sum_n \. p_n(3,3) \. q^n \, =
\, 1+ \ts q+2\ts q^2+\ts 3 \ts q^3+\ts 3\ts q^4
+\ts 3\ts q^5+\ts 3\ts q^6+\ts 2\ts q^7+\ts q^8+\ts q^9
$$
and
$$\sum_n \. p_n(3,3,1) \. q^n \, = \, q \. + \. 2\ts q^2 \. + \.
4\ts q^3 \. + \. 5\ts q^4 \. + \. 6\ts q^5 \. + \. 5\ts q^6 \. + \. 4\ts q^7 \. +
\. 2\ts q^8 \. + \. q^9.
$$
Note that even the symmetry of the last polynomial is not obvious.
For example, term $\ts 2\ts q^2\ts$ comes from two partitions each with one corner,
while $\ts 2\ts q^8\ts$ comes from one partition with
two corners (cf.~$\S$\ref{ss:fin-sym}).

\medskip

\section{Main Lemma}\label{s:main}

\noindent
For every two partitions, $\la$ and $\mu$, of size~$n$, define
$$
a_k(\lambda,\mu) \, = \,
\sum_{\alpha \vdash k, \. \beta\vdash n-k} \, c^{\lambda}_{\alpha\ts\beta}\. c^{\mu}_{\alpha\ts\beta}\.,
$$
where $\. c^{\nu}_{\pi\theta}\.$ are the \emph{Littlewood--Richardson coefficients}.  

\smallskip

\begin{lemma}[Main Lemma]\label{t:main}
For any two partitions $\lambda,\mu\vdash n$, the  sequence
$$a_0(\lambda,\mu),\ldots, a_n(\lambda,\mu)$$ is symmetric and unimodal.
\end{lemma}

We refer the reader to $\S$\ref{ss:ernesto} for additional references on this result. 


\begin{proof}
We start with \emph{Littlewood's identity}:
$$
(\circ) \qquad \quad s_{\lambda}*(s_{\pi}\ts s_{\theta}) \. = \.
\sum_{\alpha\vdash k\ts, \.\beta \vdash n-k } \.
c^{\lambda}_{\alpha\ts\beta}(s_{\alpha}*s_{\pi})(s_{\beta}*s_{\theta})\.,
$$
where $\la \vdash n$, $\pi\vdash k$ and $\theta\vdash n-k$ (see~\cite{Lit}).

If $a$ is a positive integer, then $s_{(a)}$ corresponds to the trivial representation. So we have \ts
$s_{\nu}*s_{(a)}=s_{\nu}$\ts, for all $\nu\vdash a$.
For $\pi=(k)$ and $\theta=(n-k)$, we obtain:
$$
s_{\lambda}\ts *\ts (s_{(k)}\ts s_{(n-k)}) \. = \.
\sum_{\alpha \vdash k, \. \beta \vdash n-k} \. c^{\lambda}_{\alpha\ts \beta}  \.
s_{\alpha}\ts s_{\beta} \. =\.
\sum_{\alpha\vdash k, \. \beta \vdash n-k, \. \nu \vdash n} \.
c^{\lambda}_{\alpha\ts\beta} \. c^{\nu}_{\alpha\ts\beta}\. s_{\nu}\ts.
$$
Now let $\tau=(n-k,k)$, where $k\leq n/2$.  By the Jacobi--Trudi formula,
we have:
$$
s_{\tau} \. = \. s_k \ts s_{n-k} \. - \. s_{k-1} \ts s_{n-k+1}\..
$$
We obtain:
$$s_{\lambda}*s_{\tau} \. = \.
s_{\lambda}\ts * \ts (s_k\ts s_{n-k}) \. - \. s_{\lambda}\ts * \ts (s_{k-1}\ts s_{n-k+1}) \. = \.
\sum_{\nu\vdash n} \.(\. a_k(\lambda,\nu)\ts s_{\nu} \. - \. a_{k-1}(\lambda,\nu)\ts s_{\nu}\.)\ts.
$$
Therefore, the Kronecker coefficient \ts $g(\lambda,\mu,\tau)$ \ts, which is equal to
the coefficient at $s_{\mu}$ in the expansion of \ts $s_{\lambda}*s_{\tau}$
in terms of Schur functions, is given by:
$$
g(\lambda,\mu,\tau) \. = \. a_k(\lambda,\mu) \. - \. a_{k-1}(\lambda,\mu)\ts.
$$
Since $g(\lambda,\mu,\tau)\geq 0$, the unimodality follows.  The symmetry
is clear from the definition and the symmetry of the LR coefficients, i.e. the fact that $c^{\tau}_{\al\beta}=c^{\tau}_{\beta\al}$ for any $\tau,\al,\beta$.
\end{proof}

\bigskip

\section{Special cases of the Main Lemma} \label{s:app}

\noindent
We begin with a few special cases which are obtained as corollaries to the Main Lemma when the
LR~coefficients are either~$0$ or~$1$.  We present them in
increasing order of complexity.  This is done to simplify and streamline
the exposition.

\subsection{$q$-binomial coefficients} \label{ss:app-qbin}
 We first obtain the special case $r=0$ in
Theorem~\ref{t:qbin}.  In other words, we prove unimodality of the
coefficients of $q^n$ in $\binom{m+\ell}{m}_q$.  See Section~\ref{s:fin}
for other generalizations, and~$\S$\ref{ss:app-dist} below for another approach.

\begin{cor}\label{c:qbin}
Let $p_n(\ell,m)$ be the number of partitions of $n$ which fit in the
$\ell\times m$ rectangle. Then the sequence \.$p_0(\ell,m),\ldots,p_{\ell m}(\ell,m)\.$
is symmetric and unimodal.
\end{cor}

\begin{proof}
Let $\lambda=\mu=(m^\ell)$.  Recall that
$c^{(m^\ell)}_{\alpha\beta}=1$ if $\beta$ is the complementary
partition of $\alpha$ within the $(m^\ell)$ rectangle, and is~$0$ otherwise. This can be
seen combinatorially as follows.  The SSYT and lattice permutation property enforce that
the first~$i$ rows of any skew LR~tableau contains only the first~$i$ numbers. Since
the rows in $(m^\ell)/\alpha$ are right-justified, filling them from top to bottom and
right to left, we see by induction that the rightmost numbers in row $i$ must be equal
to~$i$, and while the SSYT property forces them to be at least as many as the $(i-1)$'s
above, the lattice permutation property requires them to be exactly as many, and hence
sitting straight below the $(i-1)$'s. Continuing this way, the SSYT property enforces
at least as many $(i-1)$'s in the $i$-th row as $(i-2)$'s above them, and the lattice
permutation enforces them to be equally many, etc. This way we get a unique tableau,
as in the example below, where $m=6$, $\ell=4$ and $\al=(4,3,1)$.

\begin{center} \ytableaushort{\none\none\none\none 11,\none\none\none 122,\none 11233,122344} \end{center}

Therefore, for any $\alpha \subset (m^\ell)$, there is a unique $\beta$ giving a nonzero
LR~coefficient.  This coefficient is equal to~1, so
$$
a_n((m^\ell),(m^\ell)) \, = \, \sum_{\alpha \vdash n, \, \alpha \subset (m^\ell)}\. 1 \, = \. p_n(\ell,m)\ts.
$$
Now Lemma~\ref{t:main} implies the result.
\end{proof}

\subsection{Proof of Theorem~\ref{t:qbin}} \label{ss:app-proof}
We proceed as in the case of the $q$-binomial coefficients. We choose shapes $\lambda$
and $\mu$ such that the LR~coefficients $c^{\la}_{\al\beta}$ and $c^{\mu}_{\al\beta}$  equal~1 exactly when $\beta$
differs from the complement of $\alpha$ within $(m^\ell)$ by $r$ corners, and otherwise at least one of them is 0.

Let $\lambda = (m^\ell,1^r)$ and $\mu = (m+r,m^{\ell-1})$, i.e.~a rectangle with a column
of length $r$ attached below and the same rectangle with a row of length $r$ attached on
its right. In order for both $c^{\lambda}_{\alpha\beta}$ and $c^{\mu}_{\alpha\beta}$
to be nonzero we must have $\alpha,\beta \subset \lambda\cap \mu=(m^\ell)$.

To compute $c^{\lambda}_{\alpha\beta}$, note that the first~$\ell$ rows of LR~tableaux
in $\lambda/\alpha$ are uniquely determined, by the same argument as in the proof of
Corollary~\ref{c:qbin}. The number of $i$'s in the first $\ell$ rows of the LR~tableaux
$\lambda/\alpha$ is $m-\alpha_{\ell+1-i}=\bar{\alpha}_i$, where $\bar{\alpha}$ is
the complement of $\alpha$ within $(m^\ell)$.

The remaining $r$ rows in $\lambda$ must be filled with~$r$ distinct numbers to
preserve the SSYT property. Let these numbers be $i_1,\ldots,i_r$.
The lattice permutation property is preserved up to row $(\ell +j)$
if and only if $1+\bar{\alpha}_{i_j} \leq \bar{\alpha}_{i_j-1}$ if
$i_{j-1} \neq i_j-1$ and $j>1$, and $1+\bar{\alpha}_{i_j} \leq 1 + \bar{\alpha}_{i_j-1}$
otherwise. 
The type $\beta$ of the tableau should satisfy $\beta_i =\bar{\alpha}_i$ if $i \neq i_1,\ldots,i_r$,
and $\beta_i=\bar{\alpha}_i+1$ otherwise. This is equivalent to saying that
the type~$\beta$ of the LR~tableaux is obtained from $\bar{\alpha}$ by adding a
vertical strip of length~$r$ to it. As long as $\beta \subset (m^\ell)$,
we have $c^{\lambda}_{\alpha\beta}=1$ in this case. For all other $\beta$, we have $c^{\la}_{\al\beta}=0$.

$$\ytableausetup{centertableaux}
Y \ = \ \
\ytableaushort{\none\none\none 111, \none 11222, \none 22333,133444,{i_1},{i_2}}
$$

\medskip

\noindent
For example, for the LR~tableau~$Y$ in the figure above, we have
\ts $\alpha=(3,1,1)$, $m=6,\ell=4$, $r=2$,
and the reading word of~$Y$ is $1112221133322444331\ts i_1\ts i_2$.
In order for it to be a lattice permutation,
we can have $i_1=2$ and $i_2=4$ or $i_1=2$ and $i_2=3$, so $\beta=(6,6,5,4)$ or $\beta=(6,6,6,3)$
and while $\bar{\alpha}=(6,5,5,3)$ the vertical strip added to $\beta$
consists of a box in row 2 and~4 in the first case, or in rows 2 and 3 in the second case.

Now let $\mu=(m+r,m^{\ell-1})$. It is well known and easy to see that for
any $\mu,\alpha$ and~$\beta$,  we have
$\. c^{\mu}_{\alpha\beta}=c^{\mu'}_{\alpha'\beta'}\.$ (see e.g.~\cite{HS}).
Note that $\mu'=(\ell^m,1^r)$ has shape similar to $\lambda$,
a rectangle plus a column at the bottom.
The same argument as above applies and gives that $\beta'=\overline{\alpha'}$,
where now $\overline{\alpha'}$ is the complement of $\alpha'$ within $(\ell^m)$,
plus a vertical strip of size~$r$.  Note, however, that $\overline{\alpha'}$
is the conjugate of $\bar{\alpha}$, so applying the argument above we conclude
that~$\beta'$ is $\bar{\alpha}'$ plus a vertical strip of size~$r$.  Conjugating again,
this means that $\beta$ is $\bar{\alpha}$ plus a horizontal strip of size $r$.

 It follows that in order for both $c^{\lambda}_{\alpha\beta}\neq 0$ and
 $c^{\mu}_{\alpha\beta}\neq 0$ to hold, $\beta$ should be $\bar{\alpha}$
 plus a horizontal strip of size~$r$, and at the same time $\bar{\alpha}$
 plus a vertical strip of size $r$. This is possible if and only if the strips added are individual
squares at distinct rows and columns. In other words, $\beta$ is obtained from $\bar{\alpha}$
by adding $r$ distinct corners of $\alpha$ and for each such~$\beta$ the LR~coefficients are~$1$.
Thus, fixing $\alpha$ and summing over all possible partitions~$\beta$, we have
$$
\sum_{\beta} \. c^{\lambda}_{\alpha\beta}\ts c^{\mu}_{\alpha\beta} \. = \.
\binom{\corn(\alpha)}{r}\ts,
$$
the number of ways to select $r$ distinct corners of~$\alpha$.
Now Lemma~\ref{t:main} with $\lambda = (m^\ell,1^r)$ and $\mu = (m+r,m^{\ell-1})$ implies the result. \ $\sq$

\subsection{Partitions into distinct parts}\label{ss:app-dist}
Here we present yet another proof of Corollary~\ref{c:qbin}, which
we state in a different, but equivalent form
(see Remark~\ref{rem:dist} below).
The details of the proof are different, however.

\begin{cor}\label{c:dist}
Let $m>\ell$, and let $\. d_n(\ell,m) \.$ be the number of partitions of~$n$
into~$\ell$ distinct parts~$\le m$.  Then the sequence
$$d_\ell(\ell,m)\ts, \, d_{\ell+1}(\ell,m)\ts, \, \ldots \, , \, d_{mn}(\ell,m)
$$
is symmetric and unimodal.
\end{cor}

\begin{proof}
Let $\lambda = (m^\ell,\ell)$ and $\mu = (m+1)^\ell$. In order to have both
LR~coefficients
$c^{\lambda}_{\alpha\beta}\neq 0$ and $c^{\mu}_{\alpha\beta}\neq 0$,
the rectangular shape $\mu$ forces $\beta$ to be the complementary
of~$\alpha$ within~$\mu$, denoted~$\bar{\alpha}$. Then,
$\beta_i = m+1 -\alpha_{\ell+1-i}$, $1\leq i \leq \ell$. In this case $c^{\mu}_{\alpha\beta}=1$.
Moreover, for both LR~coefficients to be nonzero, we must have
$\alpha \subset \lambda \cap \mu = (m^\ell)$.

To compute $c^{\lambda}_{\alpha\beta}$, we construct an LR~tableau of
shape $\lambda/\alpha$ and type $\beta$. As in the previous arguments,
the first $\ell$ rows in $\lambda/\alpha$ are uniquely determined.
It is easy to see that, for $i\leq \ell$, row~$i$ of this LR~tableau has
$\alpha_{i-r}-\alpha_{i-r+1}$ numbers equal to $r$ for $r=1,\ldots,i$,
where we set $\alpha_0=m$.
Hence, in the first $\ell$ rows we have a total $m-\alpha_{\ell+1-r}$
numbers equal to~$r$. As established in the previous paragraph,
since the  LR~tableau of shape  $\lambda/\alpha$must have type~$\beta$, it follows that the numbers  $m+1-\alpha_{\ell+1-r}$ are equal to~$r$.

$$
\ytableaushort{\none\none\none\none 11,\none\none 1122,\none 12233,{*(gray)1}{*(gray)2}{*(gray)3}}
$$

\smallskip

Thus the last, $(\ell+1)$-st row of $\lambda/\alpha$ (shaded in the figure above),
must be exactly $1,2,\ldots,\ell$.
In order to preserve the SSYT property the number in row $\ell$ and column~$r$ must
be less than~$r$, which is equivalent to $\alpha_{\ell-r+1}\geq r$ for each~$r$.
In order for the final reading word to be a ballot sequence, the part of the tableaux that lies in $(m^\ell)/\alpha$
must have strictly more $r$'s than $(r+1)$'s, for $r=1,\ldots,\ell-1$, which is equivalent to
$\beta_r-1>\beta_{r+1}-1$,\ i.e.  that $\alpha$ has distinct parts.
Finally, note that together with $\alpha_i >\ell-i$, these constraints are equivalent to
$\alpha$ having $\ell$ nonzero distinct parts.  Now Lemma~\ref{t:main} implies
the result.  \end{proof}

\begin{rem}\label{rem:dist}
{\rm  Corollaries~\ref{c:dist} and~\ref{c:qbin} are
in fact equivalent, as can be shown by a natural bijection \.
$\nu \lra \al + (\ell,\ell-1,\ldots,1)$\ts.  We omit the easy details. }
\end{rem}

\bigskip

\section{Strict unimodality} \label{s:strict}

\subsection{The result}
Consider a symmetric sequence $(a_1,a_2,\ldots,a_n)$.
We say that it is \emph{strictly unimodal}, if
$$
\aligned
a_1 \, < \, a_2 \, < \, \ldots \, < \, & a_k \,  = \, a_{k+1} \, >
\, \ldots \, > \, a_n\., \qquad \text{for} \ \, n=2\ts k \\
a_1 \, < \, a_2 \, < \, \ldots \, < \, & a_k \, > \, a_{k+1} \, >
\, \ldots \, > \, a_n\., \qquad \text{for} \ \, n=2\ts k-1  
\endaligned
$$
(cf.~\cite{Med}). Strict unimodality of various partition functions
was used in~\cite[$\S$6]{PPV} to establish strict positivity of
Kronecker coefficients in a similar context.\footnote{In fact, this
paper grew out of our efforts to extend~\cite{PPV}.}  Of course, the
Main Lemma (Lemma~\ref{t:main}) does not imply strict unimodality.

In this section, we apply methods in~\cite{PPV} and reverse the
logic of the Main Lemma to obtain Theorem~\ref{t:strict} strict
unimodality of the
\emph{diagonal $q$-binomial coefficients}:
$$
\binom{2\ts m}{m}_q \, = \,\sum_{n=0}^{m^2} \. p_n(m,m) \ts q^n
$$

\begin{rem}{\rm  A direct computation shows that
strict unimodality easily fails for~$m=3$,~$4$ and~$6$
(see e.g.~\ref{ss:def-part}), but holds for $m=2$ and~$5$.
This implies that the bound $m\ge 7$ in Theorem~\ref{t:strict} is tight.
}
\end{rem}

\subsection{Partitions into distinct odd parts}
We start with the following extension of Almkvist's theorem.

\begin{thm}  \label{t:prod}
Consider the following product
$$
\cA_m(q) \, = \, \prod_{i=1}^m \. \bigl(1+q^{2\ts i-1}\bigr) \, = \, \sum_{n=0}^{m^2} \, a_n \ts q^n\ts.
$$
Then, for all \ts $m\ge 27$, the sequence $(a_{26},\ldots,a_{m^2-26})$ is symmetric and strictly unimodal.
\end{thm}

\begin{proof} Fix $m \ge 27$. The symmetry is clear.
It suffices to show that
$$
a_n \. < \. a_{n+1} \qquad \text{for all} \quad \ 26\. \le \. n \. < \. \frac{m^2-1}{2}\ts.
$$
We consider three special cases of~$n$.  First, for $n\ge 2m+1$, this was
shown in~\cite[p.~122]{A1}.

Denote by $\cQ_n$ the set of partitions of $n$ into distinct odd parts, and let
$q(n) = |\cQ_n|$.  Observe that for $n \le 2m$, we have $a_n= q(n)$.  We define
an injection $\vp: \cQ_n \to \cQ_{n+1}$ as follows.
For $\nu = (\nu_1,\ldots,\nu_\ell) \in \cQ_n$, $n\ge 3$,  let
$$
\vp(\nu) \, =
\left\{ \aligned
& (\nu_1,\ldots,\nu_\ell,1) \hskip1.05cm  \text{if} \ \ \,  \nu_\ell>1\ts, \\
& (\nu_1+2,\nu_2,\ldots,\nu_{\ell-1}) \quad \text{if} \ \ \, \nu_\ell=1\ts.
\endaligned\right.
$$
This shows that $q(n+1)\ge q(n)$.  Moreover, we have
$\nu \in \cQ_{n+1}\sm \vp(\cQ_n)$ for all partitions, s.t. $\nu_1-\nu_2=2$ and the last part is at least 3, i.e. of the form
$\nu = (2i+1,2i-1,\ldots,j)\vdash n+1$, $j\ge 3$.  For $n+1>26$,
such a partition can be taken of the form
$(2i+1,2i-1)$, $(2i+1,2i-1,5)$, $(2i+1,2i-1,7,3)$, $(2i+1,2i-1,3)$,
depending on the residue of $n$~modulo~4.  This implies that $q(n+1)> q(n)$
for all $n\ge 26$.

Now, observe that $a_{n}=q(n)$ for all $n\le 2m$, which implies that
$a_{n+1} > a_n$ for all $26 \le n \le 2m-1$.  The remaining inequality
$\ts a_{2m+1} > a_{2m}\ts$ follows from $a_{2m+1} = q(2m+1)-1$, and the additional partition
$$(2i+1,2i-1,9) \ \ \text{or} \ \, (2i+1,2i-1,7) \, \in \, \cQ_{2m+1}\sm \vp(\cQ_{2m})\ts.
$$
We omit the easy details.
\end{proof}

\begin{rem}{\rm
Note that $q(25)=q(26) = 12$ \ts (see e.g.~\cite{Slo}), so for $m\ge 13$,
we have $a_{25}=a_{26} = 12$. \ts This implies that the constant~26 in
the theorem cannot be improved.  }
\end{rem}

\subsection{Proof of Theorem~\ref{t:strict}}
We follow the approach in the proof of Corollary~6.2 in~\cite{PPV}, whose
notation we adopt.  Note that for $k\le m$ we have $p_{k}(m,m)=\pi(k)$
is the number of partitions of~$k$.  Since $\pi(k)-\pi(k-1)$ is equal
to the number of partitions with no parts~1 (see e.g.~\cite{Pak}),
we have
$$
p_{1}(m,m)\, < \,p_{2}(m,m)\, < \, \ldots \, < \, p_{m}(m,m)\ts.
$$
Assume $2\le k\le n/2$. By Lemma~\ref{t:main} and Corollary~\ref{c:qbin},
we have
$$
p_{k}(m,m) - p_{k-1}(m,m)  \. = \. g(m^m,m^m,\tau_k)\ts, \quad
\text{where} \ \ \tau_k=(n-k,k)\ts, \ \, 2 \le k \le m^2/2\ts.
$$
Therefore, reversing the logic of the proof, it suffices to show that
$$
g(m^m,m^m,\tau_k) \. \ge  \. 1\ts, \quad \text{for} \ \ \, \tau_k=(n-k,k)\ts, \ \, m \le k \le m^2/2\ts.
$$
We prove this for $m\ge 27$.  By Lemma~1.3 in~\cite{PPV}, we have \ts
$g(m^m,m^m,\tau_k) \ge 1$\ts whenever the character value
$$
\chi^{\tau_k}[2m-1,\ldots,3,1] \. \ne \. 0\ts.
$$
Following the logic of the proof of Lemma~6.1 in~\cite{PPV}, this character
is equal to the difference of partitions numbers:
$$
\chi^{\tau_k}[2\ts m-1,\ldots,3,1] \, = \, a_k \. - \. a_{k-1}\ts,
$$
where $a_k$ is as in Theorem~\ref{t:prod}.  By the theorem, for~$k\ge 27$, we have
$a_k - a_{k-1}> 0$.  In summary, for $m \ge 27$ we obtain the strict unimodality
both for $k \le m$ and $k>m$, as desired.  Finally, for $7\le m \le 26$,
we check the result by a direct computation.
\ $\sq$

\bigskip

\section{Dual version}\label{s:dual}

\noindent
In this section, we apply our general approach of using
Kronecker coefficients to prove unimodality.  Here, we use
\emph{hooks} instead of two-row Young diagrams, and then apply
the results to partitions which fit the rectangle.

\smallskip

\subsection{New unimodality result} We prove the
following version of Almkvist's theorem.

\smallskip

\begin{thm}\label{t:alm}
Consider a polynomial
$$
\cB_m(q) \, = \, \bigl(1\ts + \ts q^2 \ts + \ts q^4\ts + \ts \ldots \ts +\ts q^N \bigr) \. \cA_m(q)\.,
$$
where $N=m^2-1$ if $m$ is odd, and $N=m^2$ if $m$ is even.
Then the coefficients of $\cB_m(q)$
are symmetric and unimodal.
\end{thm}

\smallskip

\subsection{Dual version of the Main Lemma}
%
For partitions $\la, \mu \vdash n$ let
$$
b_k(\lambda,\mu)\. = \. \sum_{ \.\alpha \vdash k, \. \beta \vdash n-k}
c^{\lambda}_{\alpha\beta}\ts c^{\mu}_{\alpha'\beta} \qquad \text{and} \quad \
B_k(\lambda,\mu) \. = \. \sum_{i=0}^{\lfloor k/2 \rfloor} \. b_{k-2i}(\lambda,\mu)\ts.
$$

\smallskip

\begin{lemma}\label{hooks}
For any two partitions $\lambda,\mu\vdash n$ the sequence
$$
B_0(\lambda,\mu) \ts, \, B_1(\lambda,\mu)\ts , \, \ldots \, , \, B_n(\lambda,\mu)
$$
is weakly increasing.
\end{lemma}

\begin{proof}
We use again Littlewood's identity~$(\circ)$
from the proof of the Main Lemma, and apply it with $\pi=(1^k)$ and
$\theta = (n-k)$ to obtain
$$
s_{\lambda}*(s_{1^k}s_{n-k}) \, = \,
\sum_{\alpha \vdash k, \. \beta \vdash n-k} \.
c^{\lambda}_{\alpha\ts\beta}\ts (s_{1^k}*s_{\alpha})\ts (s_{n-k}*s_{\beta})\ts.
$$
Recall that $s_{m}*s_{\pi}=s_{\pi}$ if $\pi \vdash m$, we have $s_{1^k}*s_{\pi}=s_{\pi'}$,
where $\pi'$ is the conjugate partition. So the above identity translates as
$$
s_{\lambda}*\bigl(s_{1^k}s_{n-k}\bigr) \, = \,
\sum_{\alpha\vdash k, \.\beta \vdash n-k} \. c^{\lambda}_{\alpha\beta}
\ts s_{\alpha'}\ts s_{\beta} \, = \,
\sum_{\nu\vdash n, \. \alpha \vdash k, \. \beta \vdash n-k} \.
c^{\lambda}_{\alpha\beta}\ts c^{\nu}_{\alpha'\beta} \ts s_{\nu}
\, = \, \sum_{\nu \vdash n} \. b_k(\lambda,\nu)\ts s_{\nu}\ts.
$$
By \emph{Pieri's rule}, we have
$$s_{1^k}\ts s_{n-k}\, =\, e_k\ts h_{n-k} \, = \, s_{(n-k,1^k)} \. + \. s_{(n-k+1,1^{k-1})}\ts.
$$
Using induction on $k$, we can express the Schur function for a hook as an alternating sum:
$$
s_{(n-k,1^{k})} \, = \, e_k\ts h_{n-k} \. - \. e_{k-1}\ts h_{n-k+1} \.+\. e_{k-2}\ts h_{n-k+2}
\. - \, \ldots \, +\. (-1)^k e_0\ts h_n\ts.
$$
Thus, we have
$$
s_{\lambda}*s_{(n-k,1^k)} \, = \,
\sum_{\nu\vdash n} \, \sum_{r=0}^k \. (-1)^{r} \ts b_{k-r}(\lambda,\nu)\ts s_{\nu} \, =
\, \sum_{\nu \vdash n} \, \bigl(B_k(\lambda,\nu) \. - \. B_{k-1}(\lambda,\nu)\bigr)\ts s_{\nu}\ts.
$$
We conclude
$$
B_k(\lambda,\mu)\. - \. B_{k-1}(\lambda,\mu) \. =
\. g\bigl(\lambda,\mu,(n-k,1^k)\bigr) \. \ge \. 0\ts,
$$
as desired. \end{proof}

\smallskip

\subsection{Proof of Theorem~\ref{t:alm}}  We start with the following
combinatorial result which follows from Lemma~\ref{hooks}.

\begin{cor}
Let $w_n(m)$ be the number of self-conjugate partitions of size $(n-2i)$,
for some $i$, which fit in the $m\times m$ square. Then the sequence
$$
w_0(m)\ts, \, w_1(m)\ts, \, \ldots \, , \, w_{m^2}(m)
$$
is weakly increasing.
\end{cor}

\begin{proof}
We apply Lemma~\ref{hooks} with $\lambda=\mu=(m^m)$. As noted in the the
proof of Corollary~\ref{c:qbin}, the LR~coefficient $c^{(m^m)}_{\alpha\beta}=1$ if $\beta$
is the complementary partition of $\alpha$ within the $m\times m$ square,
 and 0 otherwise. In order for
 $c^{(m^m)}_{\alpha\beta}c^{(m^m)}_{\alpha'\beta} \neq 0$
 we must have that the complements of $\alpha$ and $\alpha'$
 within $m \times m$ are equal, which is equivalent to $\alpha=\alpha'$.
 Since for each self-conjugate $\alpha$ there is a unique complementary
 $\beta=\bar{\alpha}$ for which $\ts c^{(m^m)}_{\alpha\beta}\neq 0\ts$, we have
$$
\aligned
w_n(m) & = \, \sum_{i=1}^{\lfloor n/2\rfloor} \.
\sum_{\alpha \vdash n-2i, \.\alpha=\alpha', \.\alpha \subset (m^m) } 1
\, = \, \sum_{i=1}^{\lfloor n/2\rfloor} \. \sum_{\alpha\vdash n-2i} \.
c^{(m^m)}_{\alpha\bar{\alpha}} \ts c^{(m^m)}_{\alpha'\bar{\alpha}}\\
& = \, \sum_{i=1}^{\lfloor n/2\rfloor} \. \sum_{\alpha\vdash n-2i,\.
\beta\vdash m^2-n+2i }\. c^{(m^m)}_{\alpha\beta}\ts c^{(m^m)}_{\alpha'\beta}
\, = \, B_n(m^m,m^m)\ts.
\endaligned
$$
Now the result follows from Lemma~\ref{hooks}.
\end{proof}

Self-conjugate partitions of $n$ with largest part $\leq m$ are in a classical
bijection with partitions of $n$ into distinct odd parts $\le 2\ts m-1$,
(see e.g.~\cite{Pak}). Therefore, the Corollary implies
unimodality of the following polynomials:
$$
\bigl(1+q^2+q^4+\cdots+q^{m^2}\bigr) \. \prod_{r=1}^m (1+q^{2r-1}) \, =
\,\sum_{n=0}^{m^2} w_n(m) \ts q^n + \sum_{n=1}^{m^2} w_{m^2-n}(m)\ts q^{n+m^2}$$
for even~$m$, and
$$
\bigl(1+q^2+q^4+\cdots+q^{m^2-1}\bigr) \. \prod_{r=1}^m (1+q^{2r-1}) \, =
\,\sum_{n=0}^{m^2} w_n(m)\ts q^n \, + \, \sum_{n=1}^{m^2-1} w_{m^2-n}(m)\ts q^{n+m^2}
$$
for odd~$m$. This implies Theorem~\ref{t:alm}. \ $\sq$

\bigskip

\section{Final remarks} \label{s:fin}

\subsection{}  A combinatorial proof of unimodality of $q$-binomial coefficients
is given by O'Hara in~\cite{O} (see also~\cite{SZ,Zei}).  It would be
interesting to see if Theorem~\ref{t:qbin} can be proved by a direct
combinatorial argument.  Unfortunately, O'Hara's chain construction
argument does not seem to imply the theorem even in the case $r=1$
(cf.~$\S$\ref{ss:def-rect}).  Indeed, the value of $\corn(\al)$ is not
unimodal on the chains.  For example, the fourth chain on p.~50 in~\cite{O}
is
$$(2^2) \to (32) \to (42) \to (43) \to (43) \to (4^2) \to (4^21) \to (4^22) \to (4^23) \to (4^3)\ts,
$$
and the number of corners dips in the middle.\footnote{Note that in~\cite{O},
the author use subsets in place of partitions; the bijection is straightforward.}
Note also that O'Hara's construction does not give a symmetric chain decomposition of
the poset $L(\ell,m)$ of partitions which fit the $\ell\times m$ rectangle
(in other words, the difference between successive partitions is not always
a corner).  Existence of such decompositions remains an open problem
(see e.g.~\cite{Sta-Lef,Wen} and references therein).

\subsection{} The fact that strict unimodality of $q$-binomial coefficients
was open until now is perhaps a reflection on the lack of analytic proof
of Sylvester's theorem, as all known proofs are either algebraic or
combinatorial (see~\cite{Pro,Sta-unim}).
At the same time, our Theorem~\ref{t:qbin} is rather
mysterious; it would be nice to see a truly conceptual explanation of this
result.  While on the subject, we are curious if there is a $p$-reduction of
this result as discussed in~\cite{A2}.


\subsection{}\label{ss:fin-alm}
Theorem~\ref{t:alm} is somewhat weak, of course, and can be viewed
as both a variation on Almkvist's result as well as a statement
that the coefficients $a_n$ in $\cA_n(q)$ behave rather smoothly.
Given the sharp asymptotic results by Almkvist, it can be derived
by other means, as only unimodality of the first two and the
middle coefficients does not follow from unimodality of~$\cA_n(q)$.
We present it here as a partial triumph of algebraic methods,
as until now the analytic proof was the only result of this kind.

We should note here that it may be too much to expect an
algebraic proof of Almkvist's theorem, since $\cA_n(q)$ is
not fully unimodal, while $\cA_n(q)+q+q^{m^2-1}$ is not
combinatorially elegant.  This makes it very different from
\emph{Hughes theorem} on unimodality of
$$
\cH(t) \, = \, \prod_{i=1}^m \. \bigl(1+q^{i}\bigr)\,,
$$
which has both algebraic proofs~\cite{Hug,Sta-Lie} and an
analytic proof~\cite{OR}.  In fact, Almkvist's proof is modeled
on the Odlyzko--Richmond proof in~\cite{OR}. 

\subsection{}\label{ss:fin-sym}
In Theorem~\ref{t:qbin}, the symmetry
$$
p_n(\ell,m,r) \, = \, p_{\ell\ts m -n+r}(\ell,m,r)
$$
can be proved directly as follows.  Simply note that $p_n(\ell,m,r)$
is the number of pairs of partitions $(\al,\pi)$ such that
$\pi \ssu \al \ssu (m^\ell)$, $\al \vdash n$, and $\al/\pi$ consists
of~$r$ squares which are all (inner) corners of~$\al$.  They
are then outer corners of~$\pi$.  By taking complementary
partitions and reversing the order, we obtain pairs~$(\ov \pi, \ov \al)$
counting~$\.p_{\ell\ts m -n+r}(\ell,m,r)$\ts.


\subsection{} An important generalization of $q$-binomial coefficients
is given by $s_\la(1,q,\ldots,q^m)$, which are also known to be
unimodal~\cite[p.~137]{Mac} (see also~\cite{Kir,GOS}). The proof
goes back to Dynkin (see~\cite[p.~518]{Sta-unim}).   When
$\la=(\ell)$ or $(1^\ell)$, we get $q$-binomial coefficients
back again.

It would be nice
to find a common generalization of this result and Theorem~\ref{t:qbin}.   Note
that the most straightforward generalizations $a_k(\la)=\.$the number or
partitions~$\nu\vdash k$ which fit in the diagram~$[\la]$, is \emph{not} unimodal
in general~\cite{Stanton}.


\subsection{} Theorem~\ref{t:qbin} suggests the following generalization.
For $z\ge 1$, denote
$$
A_k(\ell,m,z) \, = \, \sum_{\al\in\cP_k(\ell,m)} \,
\frac{\Ga(\corn(\al)+z)}{\Ga(\corn(\al)+1)\ts\Ga(z)} \,,
$$
where $\Gamma(z)$ is the Gamma function.  We conjecture that $A_n(m,\ell,z)$
is unimodal.  Note that for $z\in \nn$, we have $A_k(m,\ell,z) = a_k(m,\ell,z-1)$
and the claim follows from the theorem.  See~\cite{SW} for a different one-parametric
generalization of Corollary~\ref{c:qbin}.


\subsection{} \label{ss:fin-comb-LR}
Although there are several natural combinatorial interpretations
of LR~coefficients $c_{\mu\nu}^\la$ (see e.g.~\cite{Mac,Sta}),
it is unlikely that Lemma~\ref{t:main} can be proved directly in full generality,
by an explicit surjection.  Indeed, this would give a combinatorial interpretation
of Kronecker coefficients of $g(\la,\mu,\nu)$ for $\nu=(n-k,k)$, an important
open problem whose solution is known only in a few special cases
(see~\cite{BO1,BO2,RW,Ros}).


\subsection{}\label{ss:ernesto}
After the paper was written, we learned that the  formulas in 
the proof of the Main Lemma have independently appeared in a draft version 
of~\cite{Val}, since then revised and updated.  The idea to apply these
formulas to the present unimodality results, however, is new.  

Most recently, Blasiak found a combinatorial interpretation
of the Kronecker coefficients $g(\la,\mu,\nu)$, where $\nu = (n-k,1^k)$
is a hook.  This immediately gives a combinatorial interpretation of
the difference $B_k(\lambda,\mu)-B_{k-1}(\lambda,\mu)$, as in 
Lemma~\ref{hooks}.  We use and extend this approach in~\cite{PP2}.

\subsection{}
There is yet another way to derive unimodality of $q$-binomial coefficients
(see Corollary~\ref{c:qbin}).
Recall that the Kronecker product is related to the notion of \emph{plethysm},
defined as a composition of two polynomial
representations
$$\phi: \GL(V) \rightarrow \GL(W) \quad \text{and} \quad \psi: \GL(W)\rightarrow \GL(U),
$$
giving a representation $\psi\phi:\GL(V) \rightarrow \GL(U)$, see e.g.~\cite[App.~2]{Sta}.
If the character of~$\phi$, denoted by~$f$, is expressed as a sum of
monomials via $f(x) = \sum_{\theta^i} x^{\theta^i}$ and the character of $\psi$ is~$g$,
then the character of $\psi\phi$ is given by the plethysm
$g[f] = g(x^{\theta^1},x^{\theta^2},\ldots)$. Since $\psi\phi$ is a representation
and thus decomposes into a direct sum of irreducible representations of~$\GL(V)$,
it follows that $g[f]$ is a nonnegative sum of Schur functions whenever $f$ and
$g$ are themselves nonnegative sums of Schur functions.\footnote{Another standard
notation for plethysm is $g\circ f$, see e.g.~\cite[$\S$1.8]{Mac}.}

In particular, this gives the following recipe for producing unimodal sequences.
Let $g=s_{(n-k,k)}$,  and let $f$ be any symmetric function that is a nonnegative
sum of Schur functions. Let $pl_n(\lambda,f,k)$ be the coefficient of
$s_{\lambda}(x)$ in the expansion of $h_{n-k}[f]\ts h_k[f]$
in terms of Schur functions, i.e.
$$
h_{n-k}[f] \cdot h_k[f] \, = \, \sum_{\lambda} \. pl_n(\lambda,f,k)\ts s_{\lambda}\ts.
$$
Observe that for $k\leq n/2$, we have $\de_k=pl_n(\lambda,f,k)-pl_n(\lambda,f,k-1)$
is equal to the coefficient of $s_{\lambda}$ in the expansion $s_{(n-k,k)}[f]$.
This implies that $\de_n\ge 0$, and thus
the sequence
$$
pl_n(\lambda,f,0)\., \. \ldots \. ,\. pl_n(\lambda,f,n)
$$
is symmetric and unimodal for any $\lambda\vdash n$.

For example, when $f=s_{(1,1)}$ and $\lambda=(m^{2\ell})$ this approach gives
Corollary~\ref{c:qbin} again.  We omit the details which are technical and
somewhat involved.

\subsection{}\label{ss:fin-post}  In~\cite{PP}, we generalize Theorem~\ref{t:strict}
to all all large enough $q$-binomial coefficients.  Namely, we prove that
$$
p_{1}(\ell,m)\. < \. \ldots \. < \. p_{\lfloor \ell m/2\rfloor}
\. = \. p_{\lceil \ell m/2\rceil} \. > \. \ldots \. > \. p_{\ell m-1}(\ell,m)\ts.
$$
for all $\ell,m\ge 8$.  We use a completely different approach, based on 
algebraic properties of Kronecker coefficients.  

Most recently, Shareshian found another proof of our 
Theorem~\ref{t:qbin}, which uses combinatorics of flags over~$\fq$ and 
 reduces the result to Sylvester's theorem.\footnote{Personal communication.} 


\subsection{} The \emph{log-concavity} is a stronger property than unimodality,
which appears in many applications. A sequence $a_1,\ldots,a_N$ is called 
log-concave if $a_n^2 \geq a_{n-1}a_{n+1}$ for all $2\leq n \leq N-1$.  
This property fails for $q$-binomial coefficients, but
does hold in several related contexts.  Let us single out~\cite{But} for
$q$-log-concavity of a sequence
$$\binom{n}{0}_q\., \, \binom{n}{1}_q \., \, \ldots \, , \, \binom{n}{n}_q
$$
viewed as polynomials, and \cite{Ok} for log-concavity properties of certain
LR~coefficients.  See~\cite{B2,Sta-unim} for the surveys.

\vskip.65cm

\noindent
{\bf Acknowledgements.} \ We are grateful to Jonah Blasiak, Stephen DeSalvo,
Christian Ikenmeyer, Rosa Orellana, John Shareshian, Dennis Stanton 
and Ernesto Vallejo for helpful conversations.
The first author was partially supported by the BSF and the NSF grants,
the second by the Simons Postdoctoral Fellowship.


\vskip1.1cm



{\footnotesize

\begin{thebibliography}{ddddd}\label{refpage}

\bibitem[A1]{A1}
G.~Almkvist,
Partitions into odd, unequal parts,
\emph{J.~Pure Appl. Algebra }~\textbf{38} (1985), 121--126.

\bibitem[A2]{A2}
G.~Almkvist,
Representations of $\text{\rm SL}(2,\text{\bf C})$ and unimodal polynomials,
\emph{J.~Algebra}~\textbf{108} (1987), 283--309.


\bibitem[BO1]{BO1}
C.~M.~Ballantine and R.~C.~Orellana,
On the Kronecker product $s_{(n-p,p)}\ast s_\lambda$,
\emph{Electron.~J. Combin.}~\textbf{12} (2005), RP~28, 26~pp.

\bibitem[BO2]{BO2}
C.~M.~Ballantine and R.~C.~Orellana,
A combinatorial interpretation for the coefficients in the Kronecker product $s_{(n-p,p)}\ast s_\lambda$,
\emph{S\'em. Lothar. Combin.}~\textbf{54A} (2006), Art.~B54Af, 29~pp.

\bibitem[Bla]{Bla}
J.~Blasiak,
Kronecker coefficients for one hook shape,
{\tt arXiv:}{\tt 1209.2018}.


\bibitem[B1]{B1}
F.~Brenti, Unimodal, log-concave, and P\'{o}lya frequency sequences in combinatorics,
\emph{Mem. AMS}, No.~413, 1989, 106~pp.

\bibitem[B2]{B2}
F.~Brenti,
Log-concave and unimodal sequences in algebra, combinatorics, and geometry: an update,
in \emph{Contemp. Math.}~\textbf{178}, AMS, Providence, RI, 1994, 71--89.

\bibitem[But]{But}
L.~M. Butler, The $q$-log-concavity of $q$-binomial coefficients,
\emph{J.~Combin. Theory, Ser.~A}~\textbf{54} (1990), 54--63.

\bibitem[GOS]{GOS}
F.~M.~Goodman, K.~M.~O'Hara and D.~Stanton,
A unimodality identity for a Schur function,
\emph{J.~Combin. Theory, Ser.~A}~\textbf{60} (1992), 143--146.

\bibitem[HS]{HS}
P.~Hanlon and S.~Sundaram,
On a bijection between Littlewood-Richardson fillings of conjugate shape,
\emph{J.~Combin. Theory, Ser.~A}~\textbf{60} (1992), 1--18.

\bibitem[Hug]{Hug}
J.~W.~Hughes, Lie algebraic proofs of some theorems on partitions,
in \emph{Number Theory and Algebra}, Academic Press, NY,
1977, 135--155.

\bibitem[Kir]{Kir}
A.~N.~Kirillov, Unimodality of generalized Gaussian coefficients,
\emph{C.~R.~Acad. Sci. Paris S\'{e}r.~I Math.}~\textbf{315} (1992),
no.~5, 497--501.

\bibitem[Lit]{Lit}
D.~E.~Littlewood, The Kronecker product of symmetric group representations,
\emph{J.~London Math. Soc.}~\textbf{31} (1956), 89--93.

\bibitem[Mac]{Mac}
I.~G.~Macdonald, Symmetric functions and Hall polynomials (Second ed.),
Oxford U.~Press, New York, 1995.

\bibitem[Med]{Med}
P.~Medgyessy, On the unimodality of discrete distributions,
\emph{Period. Math. Hungar.}~\textbf{2} (1972), 245--257.

\bibitem[O'H]{O}
K.~M.~O'Hara, Unimodality of Gaussian coefficients: a constructive proof,
\emph{J.~Combin. Theory, Ser.~A}~\textbf{53} (1990), 29--52.

\bibitem[Ok]{Ok}
A.~Okounkov,  Why would multiplicities be log-concave?, in
\emph{The orbit method in geometry and physics},
Birkh\"{a}user, Boston, MA, 2003, 329--347.

\bibitem[OR]{OR}
A.~M.~Odlyzko and L.~B.~Richmond,
On the unimodality of some partition polynomials,
\emph{European J.~Combin.}~\textbf{3} (1982), 69--84.

\bibitem[Pak]{Pak}
I.~Pak, Partition Bijections, a Survey, \emph{Ramanujan~J.}~\textbf{12} (2006),
5--75.

\bibitem[PP1]{PP}
I.~Pak and G.~Panova, Strict unimodality of $q$-binomial coefficients,
\emph{C.R. Math. Acad. Sci. Paris}~\textbf{351} (2013), no.~11-12, 415--418.

\bibitem[PP2]{PP2}
I.~Pak and G.~Panova, A survey of
combinatorial and complexity aspects of Kronecker coefficients,
in preparation. 

\bibitem[PPV]{PPV}
I.~Pak, G.~Panova and E.~Vallejo,
Kronecker products, characters, partitions, and the tensor square conjectures,
{\tt arXiv:}{\tt 1304.0738}.

\bibitem[Pro]{Pro}
R.~A.~Proctor,
Solution of two difficult combinatorial problems with linear algebra,
\emph{Amer. Math. Monthly}~\textbf{89} (1982), 721--734.

\bibitem[RS]{RS}
V.~Reiner and D.~Stanton,
Unimodality of differences of specialized Schur functions,
\emph{J. Algebraic Combin.}~\textbf{7} (1998), 91--107.

\bibitem[RW]{RW} J.~B. Remmel and T.~Whitehead,
On the Kronecker product of Schur functions of two row shapes,
\emph{Bull. Belg. Math. Soc. Simon Stevin}~\textbf{1} (1994), 649--683.

\bibitem[Ros]{Ros} M.~H.~Rosas,
The Kronecker product of Schur functions indexed by two row shapes of hook shapes,
\emph{J. Algebraic Combin.}~\textbf{14} (2001), 153--173.

\bibitem[Slo]{Slo}
N.~J.~A.~Sloane, A000700, in \emph{Online Encyclopedia of Integer Sequences},
\ts {\tt http://oeis.org/A000700}

\bibitem[S1]{Sta-Lie}
R.~P.~Stanley, Unimodal sequences arising from Lie algebras,
in \emph{Lecture Notes in Pure and Appl. Math.}~\textbf{57},
Dekker, New York, 1980, 127--136.

\bibitem[S2]{Sta-Lef}
R.~P.~Stanley,
Weyl groups, the hard Lefschetz theorem, and the Sperner property,
\emph{SIAM J. Algebraic Discrete Methods}~\textbf{1} (1980), 168--184.

\bibitem[S3]{Sta-unim}
R.~P.~Stanley,
Log-concave and unimodal sequences in algebra, combinatorics, and geometry,
in \emph{Ann. New York Acad. Sci.}~\textbf{576}, New York Acad. Sci., New York, 1989,
500--535.

\bibitem[S4]{Sta}
R.~P.~Stanley, \emph{Enumerative Combinatorics}, Vol.~2,
Cambridge U.~Press, Cambridge, 1999.

\bibitem[Sta]{Stanton}
D.~Stanton, Unimodality and Young's lattice,
\emph{J.~Combin. Theory, Ser.~A}~\textbf{54} (1990), 41--53.

\bibitem[SZ]{SZ}
D.~Stanton and D.~Zeilberger,
The Odlyzko conjecture and O'Hara's unimodality proof,
\emph{Proc. AMS}~\textbf{107} (1989), 39--42.

\bibitem[Syl]{Syl}
J.~J.~Sylvester,
Proof of the hitherto undemonstrated Fundamental Theorem of Invariants,
\emph{Philosophical Magazine}~\textbf{5} (1878), 178--188; reprinted
in \emph{Coll. Math. Papers}, vol.~3, Chelsea, New York, 1973,
117--126; available at \ts {\tt http://tinyurl.com/c94pphj}

\bibitem[SW]{SW}
X.-T.~Su and Y.~Wang,
Proof of a conjecture of Lundow and Rosengren on the bimodality of
$p,q$-binomial coefficients,
\emph{J.~Math.~Anal.~Appl.}~\textbf{391} (2012), 653--656.

\bibitem[Val]{Val}
E.~Vallejo, A diagrammatic approach to Kronecker squares, 
{\tt arXiv:1310.8362}. 

\bibitem[Wen]{Wen}
X.~Wen,
Computer-generated symmetric chain decompositions for $L(4,n)$ and $L(3,n)$,
\emph{Adv. Appl. Math.}~\textbf{33} (2004), 409--412.

\bibitem[Zei]{Zei}
D.~Zeilberger, Kathy O'Hara's constructive proof of the unimodality of the
Gaussian polynomials, \emph{Amer. Math. Monthly}~\textbf{96} (1989), 590–-602.

\end{thebibliography}

}
\end{document}